\documentclass[12pt, a4paper]{article}

    \usepackage{amsmath, amsthm, amssymb, amsfonts}
    \usepackage{bbm}
    \usepackage{graphicx}
    \usepackage{leftidx}

    \newtheorem{thm}{Theorem}
    
    \newtheorem{cor}[thm]{Corollary}
    \newtheorem{lem}[thm]{Lemma}

    \theoremstyle{remark}
    \newtheorem{rem}{Remark}

  \newtheoremstyle{example}{\topsep}{\topsep}%
    {}
    {}
    {\bfseries}
    {}
    {\newline}
    {\thmname{#1}\thmnumber{ #2}:\thmnote{ #3}}

  \theoremstyle{example}

    \newcommand{\gd}{\delta}
    \newcommand{\ep}{\varepsilon} 
    \newcommand{\gm}{\gamma} 
    \newcommand{\gk}{\kappa}
    \newcommand{\gl}{\lambda}

    \newcommand{\gs}{\sigma}

    \newcommand{\gz}{\zeta}

    \newcommand{\Exp}{\mathbb{E}}

    \newcommand{\ito}{It\^o}
    \newcommand{\levy}{L\'{e}vy}
    \newcommand{\gateaux}{G\^ateaux}

    \newcommand{\ind}{\mathbbm{1}}      

    \renewcommand{\Pr}{\mathbb{P}}

    \newcommand{\partf}[2]{\frac{\partial{#1}}{\partial{#2}}}

    \newcommand{\Qr}{\mathbb{Q}}

    \newcommand{\km}{K^{-}}
    \newcommand{\kp}{K^{+}}

    \newcommand{\ov}[1]{\overline{#1}}
    \newcommand{\un}[1]{\underline{#1}}

    \newcommand{\ms}{\check{\mu}}
    \newcommand{\mh}{\hat{\mu}}

    \newcommand{\Qh}{{\Qr}^z_x}
    
    \newcommand{\Eqh}{\hat{\Exp}^z_x}
    
    \newcommand{\Ph}{{\Pr}^z_x}

    \newcommand{\wta}{W_t^a}
    \newcommand{\wsa}{W_s^a}
    \newcommand{\wae}{W^{a, \ep}}
    \newcommand{\rwae}{\leftidx{_s}{W^{a, \ep}}{}}
    \newcommand{\wtae}{W_t^{a, \ep}}
    \newcommand{\wsae}{W_s^{a, \ep}}

    \newcommand{\rwm}{\leftidx{_s}{W^+}{}}
    \newcommand{\wma}{W^{+,a}}
    \newcommand{\rwma}{\leftidx{_s}{\wma}{}}

    \newcommand{\pae}{p^{a,\ep}}
    \newcommand{\aep}{A_{\ep}}

    \newcommand{\ti}[1]{\tilde{#1}}

\begin{document}

    \title{On Boundary Crossing Probabilities for Diffusion Processes}
    \author{K. Borovkov\footnote{Department of Mathematics and Statistics, University of Melbourne, K.Borovkov@ms.unimelb.edu.au} ~\&~A.N. Downes\footnote{Department of Mathematics and Statistics, University of Melbourne, A.Downes@ms.unimelb.edu.au}}
    \date{}

    \maketitle

    \begin{abstract}
    In this paper we establish a relationship between the asymptotic form of conditional boundary crossing probabilities and first passage time densities for diffusion processes. We show that, under the assumption that the conditional probability that our diffusion $(X_s, s \geq 0)$ doesn't cross an upper boundary $g(\mbox{\boldmath $\cdot$})$ prior to time $t$ given that $X_t = z$ behaves as $(a + o(1))(g(t) - z)$ as $z \uparrow g(t)$, there exists an expression for the first passage time density of $g(\mbox{\boldmath $\cdot$})$ at time $t$ in terms of the coefficient $a$ of the leading asymptotic term and the transition density of the diffusion process $(X_s)$. This assumption is shown to hold true under mild conditions. We also derive a relationship between first passage time densities for diffusions and for their corresponding diffusion bridges. Finally, we prove that the probability of not crossing the boundary $g(\mbox{\boldmath $\cdot$})$ on the fixed time interval $[0,T]$ is a (\gateaux) differentiable function of $g(\mbox{\boldmath $\cdot$})$ and give an explicit representation of the derivative.
    \end{abstract}


\section{Introduction}
\label{sec:intro}

    Consider a diffusion process $(U_s, s\geq 0)$ and an upper boundary $g(s)$ (so that $g(0) > U_0$), and, for a fixed $t>0$, denote by $(U_s^z)$ the process $(U_s)$ conditioned on $U_t = z$. The first main result of this paper is an explicit relationship between the asymptotic behaviour of the boundary crossing probabilities for this bridge diffusion process $(U_s^z)$ as $z \uparrow g(t)$ and the first passage density of the boundary at $t$ by the unconstrained process. To the best of our knowledge, this relationship has not been observed and described in the literature.

    We assume that our time-homogeneous diffusion process $(U_s)$ satisfies the stochastic differential equation
        \begin{align}
        \label{eq:orig_diff}
            dU_s = \nu(U_s) ds + \gs(U_s) dW_s, \qquad s \geq 0,
        \end{align}
    where $(W_s)$ is the standard Brownian motion and $\gs(y)$ is continuously differentiable and non-zero inside the diffusion interval (that is, the smallest interval $I \subseteq \mathbb{R}$ such that, for all $s \geq 0$, $U_s \in I$ almost surely). This can be extended to some time-inhomogeneous processes, see Remark~\ref{rem:time_inhom_ext} below. We work with the transformed process $X_s := F(U_s)$, where
        \begin{align}
        \label{eq:F_transform}
            F(y) := \int_{y_0}^y \frac{du}{\gs(u)}
        \end{align}
    for some $y_0$ inside the diffusion interval of $(U_s)$. This process satisfies the stochastic differential equation
        \begin{align}
        \label{eq:transformed_proc}
            dX_s = \mu(X_s) ds + dW_s
        \end{align}
    with $\mu(y)$ given by the composition
        \[  \mu(y) = \left( \frac{\nu}{\gs} - \frac{1}{2} \gs' \right) \circ F^{-1}(y), \]
    see e.g.\ \cite{Rogers_xx85}, p.~161. Conditions mentioned throughout refer to the transformed process $(X_s)$  and its drift coefficient $\mu$; see Remark~\ref{rem:orig_diff_dens} for further details on the relationship between results for $(X_s)$ and $(U_s)$.

    In order to establish the desired relationship between the first passage time and the asymptotic conditional crossing probabilities, we rely in Theorem~\ref{th:relate_probs} on assumption \eqref{eq:asymp_assump} describing the asymptotic form of the boundary crossing probability for the bridge process. We then show in Theorem~\ref{th:asymp_holds} that the assumption actually holds under rather mild conditions.

    These results immediately extend to first passage time densities for diffusion bridge processes. We show that if, for a given boundary, we know the first passage time density for the unconstrained process, then we immediately have the corresponding density for the bridge process as well (Theorem~\ref{thm:bridge_fpt}). For the Brownian motion case, the results of Theorems~\ref{th:relate_probs} and~\ref{thm:bridge_fpt} are closely related to those in \cite{Durbin_0385} and \cite{Durbin_etal_0692} (see Remarks~\ref{rem:durbin1} and~\ref{rem:durbin2}).

    Finally, using the methods developed when proving Theorem~\ref{th:asymp_holds}, we also show that the boundary crossing probability is \gateaux~differentiable as a function of the boundary (Theorem~\ref{thm:diff_bdy}) and give an explicit representation of the derivative.

    The paper is structured as follows. In Section~\ref{sec:main_res} we present the main results. Sections~\ref{sec:relate_proof} to~\ref{sec:proof_diff} contain the proofs of the results, Section~\ref{sec:proof_asymp} presenting a weak convergence result which may be of independent interest. Section~\ref{sec:examples} gives some examples illustrating our results in the Brownian motion case.

%
%

\section{Main Results}
\label{sec:main_res}

    We denote by $\Pr_x$ and $\Exp_x$ probabilities and expectations conditional on the process $(X_s)$ starting at the point $X_0 = x$. Where no subscript is present, either conditioning is mentioned explicitly or the process is assumed to start at zero. We also use $\ind_A$ to denote the indicator of the event $A$.

    For the process $(X_s)$ and boundary $g(s)$ with $g(0) > X_0$, define the first passage time
        \begin{align}
        \label{eq:tau_def}
            \tau := \inf\{ s>0 : X_s > g(s)\}
        \end{align}
    and the transition density
        \[  p(s,x,z) := \partf{}{z} \Pr_x (X_s \leq z). \]

    The following theorem establishes a relationship between the asymptotic form of the conditional crossing probability and the density of $\tau$.

    \begin{thm}
    \label{th:relate_probs}
        Assume that, for some $0 \leq a < b < \infty$ and $K^{\pm} \in \mathbb{R}$, the boundary $g(s)$ satisfies $g(0) >x$,
      \begin{align}
      \label{eq:lip_assump}
          -\km h \leq g(t+h) - g(t) \leq \kp h, \qquad a < t < t + h < b,
      \end{align}
        and there exists a function $f(t,x)$ continuous in $t$ such that, for $t \in (a,b)$,
            \begin{align}
            \label{eq:asymp_assump}
                \Pr_x\left(\sup_{0 \leq s \leq t} (X_s - g(s)) < 0 \,\Big|\, X_t = z\right) = (f(t,x) + o(1)) (g(t) - z)
            \end{align}
    as $z \uparrow g(t)$. Then $\tau$ has a density in the interval $(a,b)$ which is given by
            \begin{align}
            \label{eq:asymp_rel}
                p_{\tau}(t) = \frac{1}{2} f(t,x) p(t,x,g(t)), \qquad  a < t < b.
            \end{align}
    \end{thm}

    \begin{rem}
    \label{rem:durbin1}
         In \cite{Durbin_0385} it is shown that, for a standard Brownian motion $(W_t)$ (so that $x = W_0 = 0$) and a boundary $g(s)$ which is continuous in $s \in [0, t]$ and is left differentiable at $t$, the first passage time density at $t$ is given by
    \[  p_{\tau}(t) = b(t) p(t,0,g(t)), \]
where
    \begin{align}
    \label{eq:durbin_bt}
        b(t) := \lim_{s \uparrow t} \frac{1}{t-s} \Exp \left[ \ind_{ \{ \tau \geq s \} } (g(s) - W_s) \,\Big|\, W_t = g(t) \right].
    \end{align}
Thus with our Theorem~\ref{th:relate_probs} we have that, for the Brownian motion,
    \[  b(t) = \frac{1}{2} \lim_{z \uparrow g(t)} \frac{1}{g(t) - z} \Pr \left(\sup_{0 \leq s \leq t} (W_s - g(s)) < 0 \,\Big|\, W_t = z\right).    \]
Our Theorem~\ref{th:relate_probs} may be viewed as an alternative expression of the results of \cite{Durbin_0385} for Brownian motion and an extension thereof to general diffusion processes.
    \end{rem}

    \begin{rem}
    \label{rem:ken_rem}
        Representation \eqref{eq:asymp_rel} can also be viewed as an extension (to the case of general diffusion processes and curvilinear boundaries) of the well-known Kendall's identity for spectrally-negative \levy~processes $(X_s)$: if $g(s) \equiv y = $ const, $X_0 = x < y$ and $(X_s)$ has a transition density, then (see e.g.\ \cite{Kendall_xx57, Borovkov_etal_0801} and references therein)
            \begin{align}
            \label{eq:ken_rel}
                p_{\tau}(t) = \frac{y-x}{t} p(t,x,y), \qquad t>0.
            \end{align}
    \end{rem}

    \begin{rem}
    \label{rem:orig_diff_dens}
        Assume that, for the original diffusion $(U_s)$ given in \eqref{eq:orig_diff}, one has the asymptotic expression
            \[  \Pr_u \left( \sup_{0 \leq s \leq t} (U_s - g(s)) < 0 \,\Big|\, U_t = z \right) = (f(t,u)+ o(1))(g(t) - z).  \]
        Then one can easily check using Theorem~\ref{th:relate_probs} that the density of the first passage time of $(U_s)$ satisfies
            \[  p_{\tau}(t) = \frac{1}{2}f(t,u) \gs^2(g(t)) p_U(t,u,g(t)),  \]
        where $p_U(s,x,z)$ denotes the transition density for the process $(U_s)$.
    \end{rem}

    \begin{rem}
        We can extend the above to a class of diffusion processes with time-dependent drift satisfying \eqref{eq:asymp_assump}, see Remark~\ref{rem:time_inhom_ext} below.
    \end{rem}

    Under mild conditions we can establish that \eqref{eq:asymp_assump} holds for a given diffusion and boundary.

\begin{thm}
\label{th:asymp_holds}
    Let $(X_s)$ be a non-explosive diffusion satisfying~\eqref{eq:transformed_proc} with diffusion interval $\mathbb{R}$ and with $\mu \in C^1$ such that \eqref{eq:transformed_proc} has a unique strong solution and such that there exists a function $Q(y)$ satisfying
        \begin{align}
        \label{eq:Q_def}
            \mu'(y) + \mu^2(y) \geq -Q(y),  \qquad y \in \mathbb{R},
        \end{align}
and
        \begin{align}
        \label{eq:Q_bound}
            \limsup_{y \rightarrow -\infty} \frac{Q(y)}{y^2} < \frac{4}{t^2}.
        \end{align}
Let $g(s)$ be twice continuously differentiable for $s \in (a,b)$, $0 \leq a < b$. Then there exists a continuous in $t$ function $f(t,x)$ such that, for $t \in (a,b)$, \eqref{eq:asymp_assump} holds as $z \uparrow g(t)$.
\end{thm}

At this point it is worth noting that, for a linear $g(s)$ with $g(0) > x$ and for $z \leq g(t)$,
    \begin{align}
    \label{eq:bm_lin_bdy}
        \Pr_x\left( \sup_{0 \leq s \leq t} (W_s -g(s)) < 0 \,\Big|\, W_t = z \right) &= 1- \exp \left\{ -\frac{2}{t} (g(0)-x) (g(t)-z) \right\}\notag\\
            &= \frac{2}{t} (g(0) -x + o(1))(g(t)-z)
    \end{align}
as $z \uparrow g(t)$ (see e.g.\ \cite{Borodin_etal_xx02}, pp.~64--67). This serves as a motivation for the proof of Theorem~\ref{th:asymp_holds} in Section~\ref{sec:proof_asymp}, see Section~\ref{sec:examples} for further discussion of this example.

These results also extend to bridge processes, or pinned diffusions. Consider the first crossing time density of the processes $(X_s^y)$, defined as the process $(X_s)$ given in \eqref{eq:transformed_proc} conditioned to be at $y$ at time $T$. Notice that, due to the Markov property, if we restrict our attention to the time interval $[0,t]$, $t < T$, then there will be no difference between the distribution of the process $(X_s^y, 0 \leq s \leq t)$ conditioned on being at $z$ at time $t$ and that of the process $(X_s, 0 \leq s \leq t)$ conditioned on $X_t = z$. That is, for any Borel set $B \subset C[0,t]$,
    \[  \Pr_x \left(X^y_{\mbox{\boldmath $\cdot$}} \in B \,\Big|\, X_t^y = z \right) = \Pr_x \left(X_{\mbox{\boldmath $\cdot$}} \,\Big|\, X_t = z \right).  \]

In particular, this implies that for the pinned processes the function $f(t,x)$ is the same as for the unconstrained process. So, under the conditions of Theorem~\ref{th:asymp_holds},
    \begin{align*}
        \Pr_x \left(\sup_{0 \leq s \leq t} (X_s^y - g(s)) < 0 \,\Big|\, X_t^y = z\right) &= \Pr_x \left(\sup_{0 \leq s \leq t} (X_s - g(s)) < 0 \,\Big|\, X_t = z\right)\\
        &= (f(t,x)+o(1)) (g(t) - z).
    \end{align*}

Using this observation and that of Remark~\ref{rem:time_inhom_ext} at the end of Section~\ref{sec:relate_proof}, we immediately have the following result. Fix $y \in \mathbb{R}$ and denote by $p_{\tau}^y$ the density of the first hitting time \eqref{eq:tau_def} with $(X_s)$ replaced with $(X_s^y)$ and by $p^y(v,w,s,z)$ the transition density of $(X_s^y)$:
    \[  p^y(v,w,s,z) := \Pr_x \left( X_s^y \in dz \,|\, X_v^y = w \right)/dz,   \qquad 0 \leq v \leq s < T. \]

\begin{thm}
\label{thm:bridge_fpt}
    Under the conditions of Theorem~{\em \ref{th:relate_probs}}, the first crossing time density $p_{\tau}^y(t)$ of the bridge process $(X_t^y)$ satisfies
    \[  p_{\tau}^y(t) = \frac{1}{2}f(t,x) p^y(0,x,t,g(t)), \qquad a < t < b.    \]
\end{thm}

\begin{rem}
\label{rem:durbin2}
In \cite{Durbin_etal_0692}, the authors prove an extension of the results in \cite{Durbin_0385} (see Remark~\ref{rem:durbin1}) to Brownian bridges. Denote by $q^y(v,w,s,z)$ the transition density of the Brownian bridge that finishes at $y$ at time $1$. Then, for $(X_t^y)$ a pinned Brownian motion and a continuously differentiable boundary $g(s)$, \cite{Durbin_etal_0692} gives
    \[  p_{\tau}^y(t) = b(t)q^y(0,x,t,g(t)), \qquad 0 < t < 1,  \]
with $b(t)$ defined in~\eqref{eq:durbin_bt}. Thus the result of Theorem~\ref{thm:bridge_fpt} may be regarded as an alternative expression of the results in \cite{Durbin_etal_0692} for Brownian motion and an extension to more general diffusion processes.
\end{rem}

From Theorem~\ref{thm:bridge_fpt} and the representation
    \[  p^y(0,x,t,g(t)) = \frac{p(t,x,g(t))p(T-t, g(t), y)}{p(T,x,y)},  \]
we have the following relationship between the first crossing density of $(X_s)$ and that of the corresponding bridge process.

\begin{cor}
\label{cor:fpt_rels}
    Assume that, for $t \in (a,b)$, \eqref{eq:asymp_assump} holds as $z \uparrow g(t)$. Then the first crossing time densities $p_{\tau}(t)$ and $p_{\tau}^y(t)$ satisfy
        \[  p_{\tau}^y(t) = \frac{p(T-t, g(t), y)}{p(T,x,y)}p_{\tau}(t), \qquad a < t < b.  \]
\end{cor}

\begin{rem}
    A result of the form~\eqref{eq:asymp_assump} holds for Bessel processes and the constant boundary $g(t) = c >0$ as well (see e.g.\ (1.1.8) in \cite{Borodin_etal_xx02}, p.~429). Therefore we expect that Theorems~\ref{th:relate_probs},~\ref{th:asymp_holds},~\ref{thm:bridge_fpt} and Corollary~\ref{cor:fpt_rels} will also hold for diffusions with diffusion interval $(0, \infty)$.
\end{rem}

Using the approach of the present paper, we can also derive an interesting result on the sensitivity of the boundary non-crossing probability to changes in the boundary. Denote by
    \[  P(g):= \Pr_x \left(\sup_{0 \leq s \leq T}  \left(X_s - g(s) \right)  < 0 \right)    \]
the probability that our diffusion $(X_s)$ does not cross the boundary during the time interval $[0,T]$. It was shown in \cite{Downes_etal_xx08} that, under broad conditions on $(X_s)$ and $g$, the function $P(g)$ is locally Lipschitz in the uniform norm:
    \[  \big| P( g + h ) - P(g) \big| \leq C(g) \sup_{0 \leq t \leq T} | h(t)|. \]
One can expect that the function $P(g)$ will actually be (\gateaux) differentiable. The next theorem proves this conjecture. For simplicity, we state and prove the theorem assuming $T=1$ (this can easily be extended to the general case).

\begin{thm}
\label{thm:diff_bdy}
        Let $(X_s)$ be a non-explosive diffusion satisfying~\eqref{eq:transformed_proc} with diffusion interval $\mathbb{R}$ and with $\mu \in C^1$ such that \eqref{eq:transformed_proc} has a unique strong solution and such that there exists a function $Q(y)$ satisfying
        \begin{align}
        \label{eq:Q2_def}
            \mu'(y) + \mu^2(y) \geq -\ov{Q}(y), \qquad y \in \mathbb{R},
        \end{align}
and
        \begin{align}
        \label{eq:Q2_bound}
            \limsup_{y \rightarrow -\infty} \frac{\ov{Q}(y)}{y^2} < 1.
        \end{align}
Assume that $g(t)$ and $h(t)$, $0 \leq t \leq 1$, are twice continuously differentiable. Then there exists the limit
        \begin{align}
        \label{eq:deriv_expr}
            \lim_{\ep \rightarrow 0}  \;\ep^{-1} \Big[&P(g + \ep h ) - P(g) \Big] = \sqrt{\frac{2}{\pi}}  \int_0^{1} \frac{h(1-t)}{\sqrt{t}}\Pr_x(1-\tau \in dt)\notag\\
            & \times \Exp \exp \left\{ G\left( -\sqrt{t}W_{1}^+ + g(1) \right) - G(g(1-t)) + \sqrt{t}W_{1}^+  g'(1) + \ov{N}_{t}(t) \right\},
        \end{align}
    where the process $(W^+_s)$ is the Brownian meander,
            \begin{align}
            \label{eq:G_def}
                G(y) := \int_{y_0}^y \mu(z) dz
            \end{align}
  for some $y_0 \in \mathbb{R},$ $g_{0,u}(s) := g(1-u+s)$, $0 \leq s+u \leq 1$, and we set
        \begin{align*}
            \ov{N}_{u}(t) := -\frac{1}{2}\int_0^t \Big[\mu'(-\sqrt{t}W_{s/t}^+ + g_{0,u}(s)) + &\mu^2(-\sqrt{t}W_{s/t}^+ + g_{0,u}(s))\Big] ds\notag\\
            & - \sqrt{t}\int_0^{t} g''_{0,u}(s) W_{s/t}^+ ds - \frac{1}{2} \int_0^t (g'_{0,u}(s))^2 ds.
        \end{align*}
\end{thm}

%
%

\section{Proof of Theorem~\ref{th:relate_probs}}
\label{sec:relate_proof}

    Conditioning on the position of the process at time $t$, we have
        \begin{align*}
            \Pr_x(\tau \in (t,t+h)) &= \int_{-\infty}^{g(t)} \Pr_x(\tau \in (t,t+h) \,|\, X_t = z) \Pr_x(X_t \in dz)\\
                &= \int_{-\infty}^{g(t)} \Pr_x(\tau \in (t,t+h) \,|\,X_t = z) p(t,x,z) dz\\
                &= \int_{g(t) - h^{1/4}}^{g(t)} + \int_{-\infty}^{g(t) - h^{1/4}} = \int_{g(t) - h^{1/4}}^{g(t)} + o(h),
        \end{align*}
     where the last equality follows from the proof of Theorem~3.1 in \cite{Downes_etal_xx08}. Using \eqref{eq:asymp_assump}, the Markov property of the diffusion and setting $A:= \{ \sup_{t < s < t+h} (X_s - g(s)) \geq 0 \}$, we have
        \begin{align*}
            \Pr_x(\tau &\in (t,t+h) \,|\,X_t = z)\\
                & = \Pr_x\left(\sup_{0 \leq s \leq t} (X_s - g(s)) < 0 \,\Big|\, X_t = z\right) \Pr_x\left(\sup_{t < s < t+h} (X_s - g(s)) \geq 0 \,\Big|\, X_t = z\right)\\
                &=  \Big(f(t,x) (g(t) - z) + o(g(t) - z)\Big) \Pr_x(A\,|\,X_t = z), \qquad z < g(t).
        \end{align*}
     Define the functions
        \[  g^{\pm}_t (s) := g(t) \pm K^{\pm}(s-t), \]
     so by \eqref{eq:lip_assump} we have
        \[  g_t^-(s) \leq g(s) \leq g_t^+(s), \qquad t \leq s < b,  \]
     and hence, for $ 0 < h < b - t$,
        \begin{align*}
            \Pr_x\left(\sup_{t < s < t+h} (X_s - g_t^+(s)) \geq 0 \,\Big|\, X_t = z\right) \leq \Pr_x(&A\,|\,X_t = z)\\
                &\leq \Pr_x\left(\sup_{t < s < t+h} (X_s - g_t^-(s)) \geq 0 \,\Big|\, X_t = z\right).
        \end{align*}
     Define
        \[  \chi^{\pm}(z) := (g(t) - z)h^{-1/2} \pm (\mh + \km) h^{1/2}, \qquad \gm_{\pm}(z) := (g(t) - z)h^{-1/2} \pm (\ms - \kp) h^{1/2}, \]
     where
        \[  \mh:= \sup_{g(t) -h^{1/4} \leq y \leq g(t)} \mu(y), \qquad \qquad \ms:= \inf_{g(t) -h^{1/4} \leq y \leq g(t)} \mu(y).   \]
     Again using derivations from the proof of Theorem~3.1 in \cite{Downes_etal_xx08} gives, as $h \to 0$,
        \begin{align*}
            \Pr_x \bigg(\sup_{t < s < t+h} (X_s - g_t^-(s)) \geq 0 & \,\Big|\, X_t = z\bigg)\\
                &\leq \ov{\Phi}(\chi^-(z)) + e^{2(\mh + \km) (g(t) -z)} \ov{\Phi}(\chi^+(z)) + o(h)
        \end{align*}
     and
        \begin{align*}
            \Pr_x\bigg(\sup_{t < s < t+h} (X_s - g_t^+(s)) \geq 0 & \,\Big|\, X_t = z\bigg)\\
                &\geq \ov{\Phi}(\gm^-(z)) + e^{2(\ms - \kp) (g(t) -z)} \ov{\Phi}(\gm^+(z)) + o(h),
        \end{align*}
    where $\Phi$ denotes the standard normal distribution function, $\ov{\Phi}(u) = 1 - \Phi(u)$. Combining these results we have the bounds
        \begin{align}
            \Pr_x(\tau \in (t,t+h)) &\leq \int_{g(t) - h^{1/4}}^{g(t)} p(t,x,z) \big(f(t,x) (g(t) - z) + o(g(t)-z)\big)\notag\\
                &\qquad \times \left( \ov{\Phi}(\chi^-(z)) + e^{2(\mh + \km) (g(t) -z)} \ov{\Phi}(\chi^+(z)) + o(h) \right) dz,\label{eq:up_bnd}\\
            \Pr_x(\tau \in (t,t+h)) &\geq \int_{g(t) - h^{1/4}}^{g(t)} p(t,x,z) \big(f(t,x) (g(t) - z) + o(g(t)-z)\big)\notag\\
                &\qquad \times \left( \ov{\Phi}(\gm^-(z)) + e^{2(\ms - \kp) (g(t) -z)} \ov{\Phi}(\gm^+(z)) + o(h) \right) dz.\label{eq:low_bnd}
        \end{align}
    Next we will show that
        \begin{align}
        \label{eq:phi_bound}
            \int_{g(t) - h^{1/4}}^{g(t)} p(t,x,z) \big(f(t,x) (g(t) - z) &+ o(g(t)-z)\big) \ov{\Phi}(\chi^-(z)) dz\notag\\
            & \leq \frac{1}{4} h f(t,x) \sup_{g(t) - h^{1/4} \leq w \leq g(t)} p(t,x,w) + o(h),
        \end{align}
    so that the former bound in \eqref{eq:up_bnd} will yield
        \begin{align}
        \label{eq:tau_upper}
            \Pr_x(\tau \in (t,t+h)) \leq \frac{1}{2} h f(t,x) \sup_{g(t) - h^{1/4} \leq w \leq g(t)} p(t,x,w) + o(h).
        \end{align}
    The last claim uses the observation that the second $\ov{\Phi}$ term from the upper bound \eqref{eq:up_bnd} admits the same upper bound as \eqref{eq:phi_bound}, since the exponential factor is $1+ o(1)$, and the $o(h)$ term contributes a factor which itself is $o(h)$. Using the same approach, one can show that the bound \eqref{eq:low_bnd} implies that
        \begin{align}
        \label{eq:tau_lower}
            \Pr_x(\tau \in (t,t+h)) \geq \frac{1}{2} h f(t,x) \sup_{g(t) - h^{1/4} \leq w \leq g(t)} p(t,x,w) + o(h).
        \end{align}

    So we proceed to demonstrate \eqref{eq:phi_bound}. Initially ignoring the $o(g(t) - z)$ term, we have
        \begin{align*}
            J := \int_{g(t) - h^{1/4}}^{g(t)} &p(t,x,z) f(t,x) (g(t) - z) \ov{\Phi}(\chi^-(z)) dz\\
                &\leq f(t,x) \sup_{g(t) - h^{1/4} \leq w \leq g(t)} p(t,x,w) \int_{g(t) - h^{1/4}}^{g(t)} (g(t) - z)\ov{\Phi}(\chi^-(z)) dz\\
                &= D \int_{g(t) - h^{1/4}}^{g(t)} (g(t) - z) \ov{\Phi}(\chi^-(z)) dz,
        \end{align*}
    where we set $D:= f(t,x) \sup_{g(t) - h^{1/4} \leq w \leq g(t)} p(t,x,w)$. Let $y := (g(t) - z)h^{-1/2}$ and $\ov{K} := \mh + \km$. Then
        \begin{align*}
            J &\leq \frac{D}{\sqrt{2 \pi}} h \int_0^{h^{-1/4}} y \left( \int_{y-\ov{K}h^{1/2}}^{\infty} e^{-\frac{1}{2} u^2} du \right) dy\\
            &= \frac{D}{\sqrt{2 \pi}} h \left( \frac{1}{2}h^{-1/2} \int_{h^{-1/4} - \ov{K} h^{1/2}}^{\infty} e^{-\frac{1}{2} u^2} du + \frac{1}{2} \int_{-\ov{K}h^{1/2}}^{h^{-1/4}-\ov{K}h^{1/2}} e^{-\frac{1}{2}u^2} (u + \ov{K} h^{1/2})^2 du \right)\\
            &= \frac{D}{2 \sqrt{2 \pi}} h \left( h^{-1/2} \int_0^{\infty} e^{-\frac{1}{2} (w + h^{-1/4} -\ov{K} h^{1/2})^2} dw + \int_0^{h^{-1/4}}w^2 e^{-\frac{1}{2} (w - \ov{K} h^{1/2})^2} dw \right).
        \end{align*}
    The first integral in the last line is clearly $o(h)$ due to the term $-\frac{1}{2}h^{-1/2}$ in the exponential. This gives
        \begin{align*}
            J &\leq \frac{D}{2 \sqrt{2 \pi}} h \left(1 + o(1) \right) \int_0^{h^{-1/4}} w^2 e^{-\frac{1}{2} w^2} dw + o(h)\\
                &= \frac{D}{2 \sqrt{2 \pi}} h \left( \int_0^{\infty} w^2 e^{-\frac{1}{2} w^2} dw - \int_{h^{-1/4}}^{\infty} w^2 e^{-\frac{1}{2} w^2} dw \right) + o(h).\\
        \end{align*}
    Again, the last integral is clearly $o(h)$, resulting in
        \begin{align}
        \label{eq:J_final}
            J &\leq \frac{1}{4 } h f(t,x) \sup_{g(t) - h^{1/4} \leq w \leq g(t)} p(t,x,w) + o(h).
        \end{align}
    Now consider the contribution of the $o(g(t) - z)$ term to the left-hand side of \eqref{eq:phi_bound}:
        \begin{align*}
            I := \int_{g(t) - h^{1/4}}^{g(t)} o(g(t) - z) \ov{\Phi}(\chi^-(z)) dz.
        \end{align*}
    For any $\ep >0$, we can take $h$ small enough such that
        \begin{align*}
            I &\leq \ep \int_{g(t) - h^{1/4}}^{g(t)} (g(t) - z) \ov{\Phi}(\chi^-(z)) dz = \ep \left(\frac{1}{4} h + o(h)\right),
        \end{align*}
    using the same approach as to derive \eqref{eq:J_final}, so that clearly $I = o(h)$. This proves \eqref{eq:phi_bound} and hence \eqref{eq:tau_upper} and \eqref{eq:tau_lower} as well, as we noted earlier. Now the assertion of Theorem~\ref{th:relate_probs} follows due to the continuity of $f(t,x)$ in $t$ and that of $p(t,x,w)$ in $t$, $w$.

    \begin{rem}
    \label{rem:time_inhom_ext}
    If we consider a diffusion satisfying the stochastic differential equation
            \[  dU_t = \nu(t, U_t) dt + \gs(U_t) dW_t,  \]
    then the transformed process $X_t = F(U_t)$ with $F$ defined in \eqref{eq:F_transform} will again have a unit diffusion coefficient. However, in this case the drift coefficient $\mu = \mu(t,x)$ will depend on time. The proof of an analog of Theorem~\ref{th:relate_probs} in this case follows as for the homogeneous case, with $\hat{\mu}$ and $\ms$ replaced with
            \[  \mh_t:= \sup_{t \leq s \leq t+h,\; g(t) -h^{1/4} \leq y \leq g(t)} \mu(s, y), \qquad \ms_t:= \inf_{t \leq s \leq t+h,\; g(t) -h^{1/4} \leq y \leq g(t)} \mu(s, y),  \]
    respectively.
    \end{rem}

\section{Proof of Theorem~\ref{th:asymp_holds}}
\label{sec:proof_asymp}

To prove Theorem~\ref{th:asymp_holds}, we first need to establish a convergence result which is similar to those derived in \cite{Durrett_etal1_0277} and which may be of independent interest.

In what follows, we consider processes on the time interval $[0,1]$, that are sometimes pinned by their value at time $1$. This can easily be changed to the interval $[0,T]$ and processes pinned at time $T$, $T>0$.

For $a>0$ we denote by $(\wta) = (\wta, 0 \leq t \leq 1)$ the standard Brownian motion ($W_0^a = 0$) conditioned to arrive at $a$ at time $1$. For $\ep > 0$, set
    \begin{align}
    \label{eq:boundary_def}
        l_{\ep}(t) := \ep (t-1), \qquad 0 \leq t \leq 1,
    \end{align}
and for $0 \leq s \leq t \leq 1$ define the event
    \begin{align}
    \label{eq:aep}
        \aep(s,t):= \left\{ \inf_{s \leq u \leq t} (W^a_u - l_{\ep}(u) ) > 0    \right\}.
    \end{align}
Let $\aep := \aep(0,1)$ and define the conditional process $(\wtae)$ as $(\wta)$ conditioned on the event $\aep$. This will be a Markov process with transition density
    \[  \pae(s,y,t,z) := \Pr( \wtae \in dz \,|\, W_s^{a, \ep} = y)/dz   \]
for $0 \leq s < t < 1$.

\begin{thm}
\label{th:weak_conv} As $\ep \rightarrow 0$, the process $(\wtae)$ converges weakly in the space $C[0,1]$ to a process with transition density
    \begin{align}
        p^a(0,0,t,z) &:= \frac{z}{ta} \left( 1 - \exp \left\{ -\frac{2za}{1-t} \right\} \right) \Pr(\wta \in dz)\Big/ dz,\label{eq:zero_dens}\\
        p^a(s,y,t,z) &:= \frac{\left(1 - \exp \left\{ - \frac{2zy}{t-s} \right\}\right) \left(1 - \exp \left\{ - \frac{2za}{1-t} \right\}\right) }{1 - \exp \left\{ - \frac{2ay}{1-s} \right\}} \Pr(\wta \in dz \,|\, \wsa = y)\Big/dz\label{eq:non_zero_dens}
    \end{align}
for $z,y >0$ and $0 < s < t < 1$.
\end{thm}

Of course, the transition densities for $(W_s^a)$ that appear on the right-hand sides of \eqref{eq:zero_dens} and \eqref{eq:non_zero_dens} admit well-known closed form expressions (see e.g.\ \cite{Borodin_etal_xx02}, pp.~64--65).

\begin{rem}
    Note that the limiting process from Theorem~\ref{th:weak_conv} is nothing else but the Brownian meander on $[0,1]$ (see e.g.\ \cite{Durrett_etal1_0277}) conditioned to be at $a$ at time $t = 1$. This can be seen from comparing the transition densities for the two processes (for the transition density of the Brownian meander, see e.g.\ (1.1) in \cite{Durrett_etal1_0277}).
\end{rem}

\begin{proof}[Proof of Theorem~\ref{th:weak_conv}]
    The proof follows the standard scheme based on Prokhorov's theorem. First we prove convergence of the finite dimensional distributions. Since the process is Markovian, it is sufficient to prove that transition densities converge. Note that if we take the limit as $a \rightarrow 0$, these transition densities agree with those given in Theorem~5.2 of \cite{Durrett_etal1_0277} where the process is conditioned on $W_1 = 0$. More precisely, the following result holds true.
    \begin{lem}
    \label{lem:fdds} For $0 \leq s < t < 1$ and $y,z >0$,
            \[  \lim_{\ep \rightarrow 0} \pae(0,0,t,z) = p^a(0,0,t,z) \]
        and
            \begin{align*}
                \lim_{\ep \rightarrow 0} \pae(s,y,t,z) = p^a(s,y,t,z).
            \end{align*}
    \end{lem}
    \begin{proof}
        First consider $\pae(0,0,t,z)$. Using the form of the linear boundary crossing probability for a Brownian bridge (see \eqref{eq:bm_lin_bdy}) as well as the Markov property of the process, we have for $z>0$, as $\ep \rightarrow 0$,
            \begin{align*}
                \pae(&0,0,t,z) dz\\
                    &= \Pr(\wtae \in dz) = \Pr(\wta \in dz \,|\, \aep) = \Pr(\wta \in dz, \aep) /\Pr(\aep)\\
                    &= \Pr(\aep \,|\, \wta = z) \Pr(\wta \in dz) / \Pr(\aep)\\
                    &= \Pr(\aep(0, t) \,|\, \wta = z) \Pr (\aep(t,1) \,|\, \wta = z) \Pr(\wta \in dz) /\Pr(\aep)\\
                    &= \frac{ \left( 1 - \exp \left\{ -\frac{2}{t} \ep (z + \ep(1-t)) \right\} \right) \left( 1 - \exp \left\{ -\frac{2}{1-t} (z + \ep(1-t)) a  \right\} \right)} {1 - \exp \left\{ -2 \ep a \right\}} \Pr(\wta \in dz)\\
                    &= \frac{\frac{2}{t} \ep (z + \ep(1-t)) \left( 1 - \exp \left\{ -\frac{2}{1-t} (z + \ep(1-t)) a  \right\} \right)}{2 \ep a } \big( 1 + o(1) \big)\Pr(\wta \in dz)\\
                    &=  \frac{z}{ta} \left( 1 - \exp \left\{ -\frac{2za}{1-t} \right\} \right) (1+ o(1))\Pr(\wta \in dz).
            \end{align*}
        Similarly for $\pae(s,y,t,z)$:
            \begin{align*}
                \pae(&s,y,t,z) dz\\
                    &= \Pr(\wtae \in dz \,|\, \wsae = y) = \Pr( \wta \in dz, \wsa \in dy , \aep) / \Pr( \wsa \in dy, \aep)\\
                    &= \Pr( \aep \,|\, \wta = z, \wsa = y) \Pr( \wta \in dz, \wsa \in dy) / \Pr( \wsa \in dy, \aep)\\
                    &= \frac{\Pr(\aep(s,t)\,|\, \wsa = y, \wta = z) \Pr(\aep(t, 1) \,|\, \wta = z) \Pr(\wta \in dz \,|\, \wsa =y)}{\Pr( \aep(s, 1) \,|\, \wsa = y) }\\
                    &= \frac{\left(1 - \exp \left\{ - \frac{2}{t-s}(y + \ep(1-s))(z + \ep(1-t)) \right\}\right) \left(1 - \exp \left\{ - \frac{2}{1-t}(z+\ep(1-t))a \right\}\right) }{1 - \exp \left\{ - \frac{2}{1-s}a(y + \ep(1-s)) \right\}}\\
                    &\qquad \qquad \times \Pr(\wta \in dz \,|\, \wsa = y)\\
                    &= \frac{\left(1 - \exp \left\{ - \frac{2zy}{t-s} \right\}\right) \left(1 - \exp \left\{ - \frac{2za}{1-t} \right\}\right) }{1 - \exp \left\{ - \frac{2ay}{1-s} \right\}} (1+ o(1)) \Pr(\wta \in dz \,|\, \wsa = y).
            \end{align*}
    \end{proof}
Now we will prove the tightness of the family of distributions of $(\wtae)$. First we state without proof an obvious extension of Theorem~3.5 in \cite{Durrett_etal1_0277} which we will use. As usual, denote by $C[a,b]$ the space of continuous functions on $[a,b]$. Then, for $s \in (0, 1/2)$ and a function $f \in C[0,1]$, denote by $\leftidx{_s}{f}{}$ the restriction of the latter to $[s, 1-s]$: $\leftidx{_s}{f}{} \in C[s,1-s]$, $\leftidx{_s}{f(t)}{} = f(t)$ for $t \in [s, 1-s]$.

\begin{lem}
\label{th:tight_seq}
    Let $(Z_k, k = 1,2,\ldots)$ be a sequence of random elements of $C[0,1]$. Define the random elements $(\leftidx{_s}{Z}{_{k}})$ of $C[s,1-s]$ as the restrictions of $(Z_k)$ to $[s,1-s]$, $0 < s < 1/2$. Then if, for any $s \in (0,1/2)$, $(\leftidx{_s}{Z}{_k}, k = 1,2,\ldots)$ induces a tight family of distributions on $C[s,1-s]$ and, for all $\eta >0$,
        \[  \lim_{s \rightarrow 0 } \lim_{k \rightarrow \infty} \Pr \left(\sup_{0 \leq t \leq s} |Z_k(t)| \leq \eta \right)  = 1    \]
    and
        \[  \lim_{s \rightarrow 0 } \lim_{k \rightarrow \infty} \Pr \left(\sup_{1-s \leq t \leq 1} |Z_k(t) - Z_k(1)| \leq \eta \right)  = 1,    \]
    then the sequence of distributions of $(Z_k, k = 1,2,\ldots)$ in $C[0,1]$ is tight.
\end{lem}

Next we show that the conditions of Lemma~\ref{th:tight_seq} hold for our processes.
    \begin{lem}
    \label{lem:tight_restict}
        For any fixed $a > 0$ and $s \in (0,1/2)$, the family of measures induced on $C[s, 1-s]$ by $(\leftidx{_s}{\wae}{}, \ep >0)$ is tight.
    \end{lem}
    \begin{proof}
      Denote by $(W^+)$ the Brownian meander on $[0,1]$ (see e.g.\ \cite{Durrett_etal1_0277} for details). For sets $B, D \subset C[s,1-s]$ and a function $f \in C[s, 1-s]$ we use $B-D$ and $B-f$ to denote the Minkowski differences:
        \[  B - D = \left\{ g - h : g \in B, h \in D \right\}   \]
      and
        \begin{align}
        \label{eq:mink_path}
            B - f = \left\{ g - f : g \in B \right\}.
        \end{align}

      Define the random element $(\rwma)$ of $C[s, 1-s]$ as the Brownian meander $(W^+)$ conditional on $W^+_1 = a$ and restricted to $[s, 1-s]$. Note that, for all $y,z>0$ and $0 < s < 1/2$, the joint distribution of $(\rwma_s, \rwma_{1-s})$ has density $p^a(0,0,s,y) p^a(s,y,1-s,z)$. For a Borel set $B \subset C[s,1-s]$ we also have, due to the Markov property of the process,
        \begin{align}
        \label{eq:meand_inner}
            \Pr ( \rwma \in B \,|\, \rwma_s = y, \rwma_{1-s} = z) = \Pr( \rwm \in B \,|\, \rwm_s = y, \rwm_{1-s} = z).
        \end{align}
Further, by comparing the transition densities for the processes and recalling definition \eqref{eq:boundary_def}, we also have
    \begin{align}
    \label{eq:meand_our}
        \Pr \Big(\rwae \in B + l_{\ep} \,\Big|\, \wae_s = y - \ep(1-s),\,& \wae_{1-s} = z - \ep s \Big)\notag\\
            &= \Pr \Big( \rwm \in B \,\Big|\, W^+_s = y, W^+_{1-s} = z\Big).
    \end{align}

As in \cite{Durrett_etal1_0277}, for a fixed $\eta > 0$ there exists a compact set $D \subset C[s,1-s]$ such that
        \[  \Pr(\rwma \in D ) \geq 1 - \eta.    \]
The set $E:= \big(\leftidx{_s}{l_{\gd}}{}, \gd \in [0,1]\big) \subset C[s,1-s]$ is obviously also compact and so $D' := D + E$ is compact, too. Clearly, for any $\ep \in (0,1)$, $D_{\ep} := D + \leftidx{_s}{l}{_{\ep}} \subset D'$. Then we have, as $\ep \downarrow 0$, setting $I_{\ep} := (-\ep(1-s), \infty) \times (-\ep s, \infty)$, $I:=(0, \infty) \times (0, \infty)$ and using \eqref{eq:meand_our},
    \begin{align*}
        \Pr\big(\rwae \in D'\big) \geq\; &\Pr \big(\rwae \in D_{\ep}\big)\\
            =&\int_{I_{\ep}} \Pr \big(\rwae \in D_{\ep} \,\big|\, \wae_s = y, \wae_{1-s} = z \big) \Pr \big(\wae_s \in dy, \wae_{1-s} \in dz \big)\\
            = &\int_{I_{\ep}} \Pr \big(\rwm \in D_{\ep} - \leftidx{_s}{l}{_{\ep}} \,\big|\, W^+_s = y + \ep(1-s), W^+_{1-s} = z + \ep s \big)\\
                &\qquad \qquad\qquad \qquad\qquad \qquad\qquad\qquad \qquad \times  \Pr \big(\wae_s \in dy, \wae_{1-s} \in dz \big)\\
            = &\int_{I} \Pr \big(\rwm \in D \,\big|\, W^+_s = y', W^+_{1-s} = z' \big)\\
                & \qquad\qquad\qquad\qquad\qquad \times \Pr \big(\wae_s \in dy' - \ep(1-s), \wae_{1-s} \in dz' - \ep s \big)\\           \to &\int_{I} \Pr \big(\rwm \in D \,\big|\, W^+_s = y', W^+_{1-s} = z' \big) \Pr \big(\wma_s \in dy', \wma_{1-s} \in dz' \big)\\
            =&\int_{I} \Pr \big( \rwma \in D \,\big|\, \wma_s = y', \wma_{1-s} = z' \big) \Pr \big(\wma_s \in dy', \wma_{1-s} \in dz' \big)\\
            = &\Pr\big(\rwma \in D\big) > 1 - \eta,
    \end{align*}
where the convergence is justified by Scheffe's theorem (see e.g.\ Theorem~16.12 in \cite{Billingsley_xx95}, p.~215) and the second last equality uses \eqref{eq:meand_inner}. Thus there exists an $\ep_{\eta}$ such that $\Pr(\rwae \in D') > 1 - 2\eta$ for $\ep < \ep_{\eta}$ which completes the proof of the lemma.
    \end{proof}

    \begin{lem}
    \label{lem:bound_end}
        For any $\eta > 0$,
            \begin{align}
            \label{eq:left_end}
                \lim_{s \rightarrow 0 } \lim_{\ep \rightarrow 0} \Pr \left(\sup_{0 \leq t \leq s} |\wtae| \leq \eta \right)  = 1
            \end{align}
    and
            \begin{align}
            \label{eq:right_end}
                \lim_{s \rightarrow 0 } \lim_{\ep \rightarrow 0} \Pr \left(\sup_{1-s \leq t \leq 1} |\wtae - a| \leq \eta \right)  = 1.
            \end{align}
    \end{lem}
    \begin{proof} The proof uses an argument similar to the one demonstrating Lemma~5.4 in \cite{Durrett_etal1_0277}.   To establish \eqref{eq:left_end}, first note that
        \[  \Pr \left(\inf_{0 \leq t \leq s} \wtae \geq -\ep \right) = 1,   \]
    and so we just need to consider
        \begin{align*}
            \Pr \left( \sup_{0 \leq t \leq s} \wtae \leq \eta \right) &= \Pr\left( \sup_{0 \leq t \leq s} \wta \leq \eta \,\Big|\,\aep \right) = \Pr\left( \sup_{0 \leq t \leq s} \wta \leq \eta; \aep \right)\Big/\Pr(\aep).
        \end{align*}
    Denote by $q(t,y,z) = \partf{}{z} \Pr_y (W_t \leq z)$ the transition density for the Brownian motion process. By conditioning on the value of the process at time $s$, using the Markov property and the known closed form expressions for both the joint distribution of the maximum and minimum of the Brownian bridge (see e.g.\ (1.15.8) of \cite{Borodin_etal_xx02}) and the distribution of the maximum of the Brownian bridge given in \eqref{eq:bm_lin_bdy}, we obtain
        \begin{align*}
            \Pr \bigg( &\sup_{0 \leq t \leq s} \wta \leq \eta; \aep \bigg)\\
                &= \int_{-\ep(1-s)}^{\eta} \Pr\left( \sup_{0 \leq t \leq s} \wta \leq \eta; \aep(0,s) \,\Big|\, \wsa = z \right) \Pr \big(\aep(s,1) \,\big|\, \wsa = z\big) \Pr(\wsa \in dz)\\
                &\geq \int_{-\ep(1-s)}^{\eta} \Pr\left( \sup_{0 \leq t \leq s} \wta \leq \eta, \inf_{0 \leq t \leq s} \wta >-\ep(1-s) \,\Big|\, \wsa = z \right) \Pr\big(\aep(s,1) \,\big|\, \wsa = z\big)\\
                & \qquad\qquad\qquad\qquad\qquad\qquad\qquad\qquad\qquad\qquad\qquad\qquad\qquad\qquad\qquad \times \Pr(\wsa \in dz)\\
                &= \int_{-\ep(1-s)}^{\eta} \sum_{k = -\infty}^{\infty} \Big[q\Big(s, 0, z+ 2 k [\eta + \ep(1-s)]\Big) - q\Big(s, 0, z + 2k (\eta + \ep(1-s)) + 2 \ep(1-s)\Big) \Big]\\
                    &\qquad \times \left( 1 - \exp\left\{ -\frac{2a}{1-s}(z+\ep(1-s)) \right\} \right)\big(q(s,0,z)\big)^{-1} \Pr(\wsa \in dz).
        \end{align*}
    Now dividing both sides by
        \[  \Pr(\aep) = 1 - \exp \{-2 \ep a\}   \]
    (cf.\ \eqref{eq:bm_lin_bdy}) and using L'H\^opital's rule, the dominated convergence theorem and the fact that $W_s^a \sim N(as, s(1-s))$, we obtain
        \begin{align*}
            \lim_{\ep \rightarrow 0} \Pr \left( \sup_{0 \leq t \leq s} \wtae \leq \eta \right) &\geq \int_0^{\eta} \sum_{k = -\infty}^{\infty} \frac{1-s}{as} (2 \eta k +z) e^{-2 \eta k (\eta k + z)/s} \\
            &\qquad\qquad\qquad \times \left(1 - \exp \left\{ - \frac{2za}{1-s} \right\} \right) q(s(1-s), as, z)dz.
        \end{align*}
    As in the proof of Lemma~5.4 in \cite{Durrett_etal1_0277}, the integral of the term with $k = 0$ approaches $1$ as $s \rightarrow 0$. The contribution of the terms for $k \geq 1$ is clearly positive, so to complete the proof it is sufficient to show that the absolute value of the contribution from the terms with $k < 0$ tends to $0$ as $s \rightarrow 0$. Setting $m = -k$, the absolute value for these terms is bounded above by the expression
        \begin{align}
        \label{eq:pos_k_bound}
            \frac{2 \eta(1-s)}{as} \int_0^{\eta} \left(1 - \exp \left\{ - \frac{2za}{1-s} \right\} \right) q(s(1-s),as, z) \sum_{m =1}^{\infty} m e^{ -2 \eta m (\eta m -z)/s }dz.
        \end{align}
    The contribution of the term with $m=1$ is bounded above by
        \begin{align*}
            &\frac{2 \eta(1-s)}{as} \int_0^{\eta} q(s(1-s),as, z) e^{ -2 \eta (\eta -z)/s }dz\\
                & \quad= \frac{2 \eta(1-s)}{as} e^{-2\eta (\eta - a)} \left(\Phi\left(-\frac{\eta (1-2s) + a s}{\sqrt{s(1-s)}} \right) - \Phi\left(-\frac{2\eta (1-s) + a s}{\sqrt{s(1-s)}} \right) \right) = o(1)
        \end{align*}
    as $s \rightarrow 0$. The same approach shows that the $m=2$ term also gives a contribution which is $o(1)$ as $s \rightarrow 0$. For the remaining terms, observe that
        \[  \sum_{m=3}^{\infty} m e^{ -2 \eta m (\eta m -z)/s } \leq \sum_{m=3}^{\infty} m e^{ -2 \eta^2 (m-1)^2/s },   \qquad 0\leq z \leq \eta.   \]
    We may assume without loss of generality that $s$ is small enough such that the summand in the above expression is a strictly decreasing function of $m$ for $m \geq 3$. We then have that
        \begin{align*}
            \sum_{m=3}^{\infty} m e^{ -2 \eta^2 (m-1)^2/s } &\leq \int_2^{\infty} w e^{-2 \eta^2 (w-1)^2/s}dw\\
            & = \frac{1}{4 \eta^2} \left( s e^{-2 \eta^2/s} + 2\eta \sqrt{2 \pi s} \Phi\left(\frac{-2\eta}{\sqrt{s}}\right) \right) \leq \frac{1}{2 \eta^2} s e^{-2 \eta^2/s},
        \end{align*}
    where the last relation uses Mill's inequality. Thus the absolute value of the sum of all terms for $m \geq 3$ in \eqref{eq:pos_k_bound} is bounded above by
        \[  \frac{(1-s)}{\eta a} e^{-2 \eta^2/s} \int_0^{\eta} \left(1 - \exp \left\{ - \frac{2za}{1-s} \right\} \right) q(s(1-s),as, z) dz \leq \frac{(1-s)}{\eta a} e^{-2 \eta^2/s} = o(1) \]
    as $s \rightarrow 0$. This completes the proof of \eqref{eq:left_end}.

        Now consider \eqref{eq:right_end}. We have
            \begin{align}
            \label{eq:end_convergence}
                \Pr \bigg(\sup_{1-s \leq t \leq 1} &|\wtae - a| \leq \eta \bigg)\notag\\
                    & = \Pr \left( \inf_{1-s \leq t \leq 1} \wtae -a \geq -\eta, \sup_{1-s \leq t \leq 1} \wtae - a \leq \eta \right)\notag\\
                    &= \Pr \left( \inf_{1-s \leq t \leq 1} \wta \geq -\eta + a, \sup_{1-s \leq t \leq 1} \wta \leq \eta + a; \aep \right)/\Pr(\aep) =: J_1/J_2.
            \end{align}
        For $J_1$, conditioning on the value of $W_{1-s}^a$ gives (without loss of generality, we take $a - \eta > 0$ --- if this were not the case, we would replace the lower limit of integration with $0$)
            \begin{align*}
                J_1 &= \int_{a - \eta}^{a + \eta} \Pr \left( \inf_{1-s \leq t \leq 1} W_t \geq -\eta + a, \sup_{1-s \leq t \leq 1} W_t \leq \eta + a; \aep \,\bigg|\, W_{1-s} = y, W_1 = a\right)\\
                &\qquad \qquad\qquad \qquad\qquad \qquad\qquad \qquad\qquad \qquad\qquad \qquad\qquad  \times \Pr(W_{1-s} \in dy \,|\, W_1 = a)\\
                    &= \int_{a - \eta}^{a + \eta} \Pr \left( \inf_{1-s \leq t \leq 1} W_t \geq -\eta + a, \sup_{1-s \leq t \leq 1} W_t \leq \eta + a \,\bigg|\,W_{1-s} = y, W_1 = a\right)\\
                    & \qquad \qquad\qquad \qquad\qquad \qquad\qquad \times  \Pr \left(\aep(0,1-s) \,|\, W_{1-s} = y \right) \Pr(W_{1-s} \in dy \,|\, W_1 = a).
            \end{align*}
        Since $\Pr \left(\aep(0,1-s) \,|\, W_{1-s} = y \right) = 1 - \exp\left\{-2\ep ( y + \ep(1-s))/(1-s) \right\}$ and $J_2 = 1 - \exp \left\{ -2 \ep a \right\}$ (see \eqref{eq:bm_lin_bdy}), using the dominated convergence theorem in \eqref{eq:end_convergence} gives
        \begin{align*}
            \lim_{\ep \rightarrow 0} \Pr \bigg(\sup_{1-s \leq t \leq 1} &|\wtae - a| \leq \eta \bigg) = \int_{a - \eta}^{a + \eta} \Pr(W_{1-s} \in dy \,| \, W_1 = a) \frac{y}{a(1-s)}  \\
            & \qquad \times  \Pr \left( \inf_{1-s \leq t \leq 1} W_t \geq -\eta + a, \sup_{1-s \leq t \leq 1} W_t \leq \eta + a \,|\,W_{1-s} = y, W_1 = a\right).
        \end{align*}
        This clearly tends to $1$ as $s \rightarrow 0$, completing the proof of Lemma~\ref{lem:bound_end}.
    \end{proof}
Combining Lemmas~\ref{th:tight_seq}--\ref{lem:bound_end} completes the proof of Theorem~\ref{th:weak_conv}.

\end{proof}

We now prove Theorem~\ref{th:asymp_holds}.

\begin{proof}[Proof of Theorem~\ref{th:asymp_holds}]

The proof of Theorem~\ref{th:asymp_holds} can be split into three steps. In the first step, we change the measure to express the conditional probability of the event on the left-hand side of \eqref{eq:asymp_assump},
    \begin{align}
    \label{eq:orig_A}
        A:= \left\{\sup_{0 \leq s \leq t} (X_s - g(s)) <0 \right\},
    \end{align}
in terms of the expectation of a (nice) functional of the Brownian bridge process (the Brownian motion starting at $x$ at time $0$ and pinned at $z$ at time $t$), using the approach employed in \cite{Baldi_etal_0802} and later in \cite{Downes_etal_xx08}. Now the underlying ``pre-conditional'' process (the Brownian motion) is space-homogeneous, and therefore we can ``straighten'' the boundary $g(s)$ by switching from the canonical process $(X_s)$ to the process $(X_s - g(s))$ (the second step). In the transformed space, the original set $A$ becomes
    \begin{align}
    \label{eq:new_A}
        \ti{A}:= \left\{\sup_{0 \leq s \leq t} X_s <0 \right\},
    \end{align}
and the conditional process can be thought of as obtained from the diffusion $(\ti{X}_s)$ with
    \begin{align}
    \label{eq:X_tilde}
        d\ti{X}_s = -g'(s) ds + d\ti{W}_s
    \end{align}
($(\ti{W}_s)$ being a Brownian motion), $\ti{X_0} = \ti{x} := x - g(0) <0$, conditioned on $\ti{X}_t = \ti{z} := z - g(t) <0$. We again change measure (now in the transformed space) to express the desired expectation in terms of another one, again for a Brownian bridge process --- now starting at $\ti{x}$ at time $0$ and finishing at $\ti{z}$ at time $t$ --- over the event $\ti{A}$. In the third step, we re-write the latter expectation as the product of the (known) probability of the Brownian bridge to stay below zero (this factor will have the desired behaviour $a \ti{z}(1+ o(1))$ as $\ti{z} \rightarrow 0$) and the conditional expectation, where the conditioning now includes the event $\ti{A}$ as well. It remains to observe that the Brownian bridge is `time-reversible' (if $(Y_s, 0 \leq s \leq t)$ is a Brownian bridge `pinned' at times $s= 0$ and $s=t$, then $(Y_{t-s}, 0 \leq s \leq t)$ is also a Brownian bridge pinned at these same times), and so the behaviour of the last conditional expectation as $\ti{z} \uparrow 0$ can be found using our Theorem~\ref{th:weak_conv}.

Now we will make the above-outlined argument more precise.

\textbf{Step 1:} Let $\Qr_x$ denote the law of the Brownian motion $(W_s)$ in $C[0,t]$ with $W_0 = x$ and $\Qr_x^z$ the law of $(W_s)$ with $W_0=x$ and $W_t = z$. We begin by recalling the following result from \cite{Baldi_etal_0802}.

    \begin{lem}
    \label{lem:bridge_prob}
    Let $\Ph$ denote the law of $(X_s)$ governed by \eqref{eq:transformed_proc} starting at $X_0 = x$ and pinned by $X_t = z$. Then, for any $B \in \gs\left(X_u: u \leq t \right)$,
            \[ \Ph(B) = \frac{q(t, x, z)}{p(t, x, z)} e^{G(z) - G(x)}   \Eqh\left[e^{N(t)}\ind_{ B}\right], \]
        where $q(t,x,z)$ is the transition density for the Brownian motion and $\Eqh$ denotes expectation with respect to the probability $\Qh$, $G(y)$ is defined in \eqref{eq:G_def} and
            \begin{align}
            \label{eq:Nt_def}
                N(t) := -\frac{1}{2}\int_0^t \left(\mu'(X_u) + \mu^2(X_u)\right) du.
            \end{align}
    \end{lem}

 Note that the notation of the above lemma tacitly assumes that $(X_s)$ is a canonical process on the sample space $C[0,t]$, and we continue with this assumption throughout the proof of Theorem~\ref{th:asymp_holds}.

 Hence for the event $A$ defined in \eqref{eq:orig_A} we have
    \begin{align}
    \label{eq:old_rn_deriv}
        \Pr_x\left(\sup_{0 \leq s \leq t} (X_s - g(s)) < 0 \,\Big|\, X_t = z\right) = \Ph(A) = \frac{q(t, x, z)}{p(t, x, z)} e^{G(z) - G(x)}   \Eqh\left[e^{N(t)}\ind_{ A}\right].
    \end{align}
Note that the changes of measure used here and in the sequel using Girsanov's theorem are justified since under these measures $(X_s)$ is non-explosive and $\mu(y)$ (and $g(s)$) is locally bounded, implying the conditions of Theorem~7.19 of \cite{Liptser_etal_xx01}, as well as Assumptions (I) an (II) of Theorem 7.18 (with our assumption of the existence of a unique strong solution to \eqref{eq:transformed_proc}), hold.

\textbf{Step 2:} We now transform the space, defining $\psi:C[0,t] \to C[0,t]$ by $\psi(f) = f - g$. This mapping induces a new measure $\ti{\Qr}_{\ti{x}}$ on the space: for a Borel set $B \subset C[0,t]$,
    \[  \ti{\Qr}_{\ti{x}}(B) = \Qr_x(B + g),    \]
again using the Minkowski difference notation defined in \eqref{eq:mink_path} and $\ti{x} = x - g(0)$. In the same way $\psi$ also induces the measure $\ti{\Qr}_{\ti{x}}^{\ti{z}}$ (with expectation $\ti{\Exp}_{\ti{x}}^{\ti{z}}$) from $\Qr_x^z$. This gives
    \[  \Eqh\left[ e^{N(t)}\ind_{A} \right] = \ti{\Exp}_{\ti{x}}^{\ti{z}}\left[ e^{\ti{N}(t)}\ind_{\ti{A}} \right], \]
where
    \[  \ti{N}(t) := -\frac{1}{2}\int_0^t \left[\mu'(X_s + g(s)) + \mu^2(X_s + g(s))\right] ds  \]
and $\ti{A}$ is given in $\eqref{eq:new_A}$. Note that under the measure $\ti{\Qr}_{\ti{x}}^{\ti{z}}$ the canonical process $(X_s)$ is no longer a Brownian motion. However, as stated previously, we can think of the conditional process as being obtained from the diffusion $(\ti{X}_s)$ defined in \eqref{eq:X_tilde}. We change measure again such that under this new measure, $\ov{\Qr}_{\ti{x}}$, the process $(X_s)$ is again a Brownian motion starting at $X_0 = \ti{x}$. By Girsanov's theorem,
    \begin{align*}
        \tilde{\gz}_t := \frac{d\ti{\Qr}_{\ti{x}}}{d\ov{\Qr}_{\ti{x}}} = \exp \left\{- \int_0^t g'(s) dX_s - \frac{1}{2}\int_0^t (g'(s))^2 ds \right\}.
    \end{align*}
Using \ito's formula we have
    \[  \int_0^t g'(s) dX_s = X_t g'(t) - \int_0^t g''(s) X_s ds.   \]
This gives
    \[  \tilde{\gz}_t = \exp \left\{ -X_t g'(t) + \int_0^t g''(s) X_s ds - \frac{1}{2}\int_0^t (g'(s))^2 ds \right\}.   \]
Let $\ov{\Qr}_{\ti{x}}^{\ti{z}}$ (with corresponding expectation $\ov{\Exp}_{\ti{x}}^{\ti{z}}$) denote the law $\ov{\Qr}_{\ti{x}}$ conditioned on $X_t = \ti{z}$; this is clearly the distribution of the Brownian bridge on $[0,t]$ pinned at $\ti{x}$ at time $s=0$ and at $\ti{z}$ at time $s=t$. Applying the same reasoning that leads to the assertion of Lemma~\ref{lem:bridge_prob} (see \cite{Baldi_etal_0802}) then gives
    \begin{align}
    \label{eq:second_meas_chng}
        \Eqh\left[ e^{\ti{N}(t)}\ind_{A} \right] &= \frac{\tilde{p}(0,\ti{x},t,\ti{z})}{q(t,\ti{x},\ti{z})} \ov{\Exp}_{\ti{x}}^{\ti{z}} \left[ \ti{\gz}_t e^{N(t)}  \ind_{\tilde{A}} \right]= \frac{\tilde{p}(0,\ti{x},t,\ti{z})}{q(t,\ti{x},\ti{z})}e ^{-\tilde{z} g'(t)} \ov{\Exp}_{\ti{x}}^{\ti{z}} \left[ e^{\ov{N}(t)}  \ind_{\tilde{A}} \right],
    \end{align}
where $\tilde{p}(0,\ti{x},t,\ti{z})$ is the transition density of the (time-inhomogeneous) process $(X_s)$ under $\ov{\Qr}_{\ti{x}}$, and we set
    \begin{align*}
        \ov{N}(t) &:= -\frac{1}{2}\int_0^t \Big[\mu'\left(X_s + g(s)\right) + \mu^2\left(X_s + g(s)\right)\Big] ds\notag\\
        &  \qquad \qquad \qquad \qquad \qquad \qquad  + \int_0^t g''(s) X_s ds - \frac{1}{2}\int_0^t (g'(s))^2 ds.
    \end{align*}
Combining \eqref{eq:old_rn_deriv} and \eqref{eq:second_meas_chng} then gives
    \[  \Pr_x\left(\sup_{0 \leq s \leq t} (X_s - g(s)) < 0 \,\Big|\, X_t = z\right) = \frac{q(t, x, z)}{p(t, x, z)} \frac{\tilde{p}(0,\ti{x},t,\ti{z})}{q(t,\ti{x},\ti{z})} e^{G(z) - G(x)-\tilde{z} g'(t)}  \ov{\Exp}_{\ti{x}}^{\ti{z}} \left[ e^{\ov{N}(t)}  \ind_{\tilde{A}} \right].    \]

\textbf{Step 3:} Now condition on the event $\tilde{A}$. From \eqref{eq:bm_lin_bdy}, we have $\ov{\Qr}_{\ti{x}}^{\ti{z}}(\ti{A}) = 1 - \exp\{-2 \ti{x} \ti{z}/t \}$, which gives
    \begin{align}
    \label{eq:final_prob}
        \Pr_x\bigg(\sup_{0 \leq s \leq t} &(X_s - g(s)) < 0 \,\Big|\, X_t = z\bigg)\notag\\
            & = \frac{q(t, x, z)}{p(t, x, z)}\frac{\tilde{p}(0,\tilde{x},t,\tilde{z})}{q(t,\ti{x},\ti{z})} e^{G(z) - G(x)-\tilde{z} g'(t)}  \left(1 - \exp\left\{-\frac{2}{t} \tilde{x} \tilde{z} \right\} \right) \ov{\Exp}_{\ti{x}}^{\ti{z}} \left[ e^{\ov{N}(t)} \,\Big|\, \tilde{A}\right]\notag\\
            & = \frac{q(t, x, z)}{p(t, x, z)}\frac{\tilde{p}(0,\tilde{x},t,\tilde{z})}{q(t,\ti{x},\ti{z})} e^{G(z) - G(x)-\tilde{z} g'(t)} \frac{2}{t} \ti{x} \ti{z} \ov{\Exp}_{\ti{x}}^{\ti{z}} \left[ e^{\ov{N}(t)} \,\Big|\, \tilde{A}\right] \big( 1 + o(1) \big)
    \end{align}
as $z \uparrow g(t)$ (i.e.\ $\ti{z} \uparrow 0$). Due to the above mentioned time-reversal and symmetry properties of the Brownian bridge, we obtain that
    \begin{align}
    \label{eq:rev_exp}
        \ov{\Exp}_{\ti{x}}^{\ti{z}} \left[ e^{\ov{N}(t)} \,\Big|\, \tilde{A}\right] = \ov{\Exp}_{\ti{z}}^{\ti{x}} \left[ e^{\un{N}(t)} \,\Big|\, \tilde{A}\right],
    \end{align}
where
    \begin{align}
    \label{eq:under_N}
        \un{N}(t) = -\frac{1}{2} \int_0^t \big[\mu'(X_s + g(t-s))+& \mu^2(X_s + g(t-s)) \big] ds\notag\\
        & + \int_0^t g''(t-s) X_s ds - \frac{1}{2} \int_0^t (g'(t-s))^2 ds.
    \end{align}
Thus we are now in the situation of Theorem~\ref{th:weak_conv} (with an appropriate
change of scale), conditioning on our event $\ti{A}$ being equivalent to conditioning
on the event $\aep(0,1)$ (see \eqref{eq:aep}) in the theorem (note that the initial
value is now $\tilde{z} \uparrow 0$). To complete the proof, we need to show that the
expectation \eqref{eq:rev_exp} converges to a finite limit as $\tilde{z} \uparrow 0$.
For a fixed large $H>0$, let $D:= \{ \inf_{0 \leq s \leq t} X_s \leq -H\}$ and, as
usual, $D^c$ denote the complement event. We have
    \begin{align}
    \label{eq:split_exp}
        \ov{\Exp}_{\ti{z}}^{\ti{x}} \left[ e^{\un{N}(t)} \,\Big|\, \tilde{A}\right] = \ov{\Exp}_{\ti{z}}^{\ti{x}} \left[ e^{\un{N}(t)} \ind_D\,\Big|\, \tilde{A}\right] + \ov{\Exp}_{\ti{z}}^{\ti{x}} \left[ e^{\un{N}(t)} \ind_{D^c} \,\Big|\, \tilde{A}\right].
    \end{align}
In the integrand in the second term on the right-hand side of \eqref{eq:split_exp},
$\un{N}(t)$ is clearly a bounded continuous function of $X_{\mbox{\boldmath $\cdot$}}$
(in the uniform topology) on $\tilde{A}.$ Since $D$ has zero boundary under the
limiting distribution, we see, using the weak convergence result of
Theorem~\ref{th:weak_conv}, that this term converges to a finite limit as $\tilde{z}
\uparrow 0$. We complete the proof of Theorem~\ref{th:asymp_holds} by showing that, as
$H \rightarrow \infty$, the first term on the right-hand side of \eqref{eq:split_exp}
converges to $0$ (which is basically equivalent to uniform integrability of
$e^{\un{N}(t)}$ under the distributions $\ov{\Qr}_{\ti{z}}^{\ti{x}}\big(\mbox{\boldmath
$\cdot$} \,\big|\, \tilde{A}\big)$, $\tilde{z} \in (-1, 0)$).

We begin by considering the distribution of the minimum of $(X_s)$ for $s \in [0,t]$, under the measure $\ov{\Qr}_{\ti{z}}^{\ti{x}}$ and conditional on $\tilde{A}$. Define $\gm:= \inf\{s>0 : X_s = \tilde{x} \}$. Under $\ov{\Qr}_{\ti{z}}^{\ti{x}}$, we clearly have that $\gm \leq t$ a.s. Then for $y<\tilde{x}$, observe that
    \begin{align}
    \label{eq:bound_sup}
        \ov{\Qr}_{\ti{z}}^{\ti{x}} \left( \inf_{0 \leq s \leq t} X_s \leq y \, \bigg| \, \tilde{A} \right) = \int_0^t \ov{\Qr}_{\ti{z}}^{\ti{x}} (\gm \in du \, | \, \tilde{A} ) \ov{\Qr}_{\ti{z}}^{\ti{x}} \left( \inf_{0 \leq s \leq t} X_s \leq y \, \bigg| \, \tilde{A}; \gm = u \right).
    \end{align}
Using the strong Markov property of $(X_s)$ and \eqref{eq:bm_lin_bdy} (recalling that, under $\ov{\Qr}_{\ti{z}}^{\ti{x}}$, the process $(X_s)$ is a pinned Brownian motion), we have that
    \begin{align*}
        \ov{\Qr}_{\ti{z}}^{\ti{x}} \bigg( \inf_{0 \leq s \leq t} &X_s \leq y \, \bigg| \, \tilde{A}; \gm = u \bigg) = \ov{\Qr}_{\ti{z}}^{\ti{x}} \bigg( \inf_{u \leq s \leq t} X_s \leq y \, \bigg| \, \tilde{A}; \gm = u \bigg)\\
            & = \ov{\Qr}_{\ti{z}}^{\ti{x}} \left( \inf_{u \leq s \leq t} X_s \leq y , \sup_{u \leq s \leq t} X_s < 0 \, \bigg| \, X_u = \tilde{x} \right) \bigg/ \ov{\Qr}_{\ti{z}}^{\ti{x}} \left( \sup_{u \leq s \leq t} X_s < 0 \, \bigg| X_u = \tilde{x} \right)\\
            & \leq \ov{\Qr}_{\ti{z}}^{\ti{x}} \left( \inf_{u \leq s \leq t} X_s \leq y \, \bigg| \, X_u = \tilde{x} \right) \bigg/ \ov{\Qr}_{\ti{z}}^{\ti{x}} \left( \sup_{u \leq s \leq t} X_s < 0 \, \bigg| X_u = \tilde{x} \right)\\
            & = \frac{\exp \left\{ - \frac{2}{t-u}(y-\tilde{x})^2\right\}}{1 - \exp \left\{ -\frac{2}{t-u} \tilde{x}^2 \right\} } \leq \frac{e^{-2 (y-\tilde{x})^2/t}}{1- e^{-2 \tilde{x}^2/t}}.
    \end{align*}
Hence an upper bound for \eqref{eq:bound_sup} is given by
    \begin{align}
    \label{eq:inf_bound}
        \ov{\Qr}_{\ti{z}}^{\ti{x}} \left( \inf_{0 \leq s \leq t} X_s \leq y \, \bigg| \, \tilde{A} \right) \leq \frac{e^{-2 (y-\tilde{x})^2/t}}{1- e^{-2 \tilde{x}^2/t}} \int_0^t \ov{\Qr}_{\ti{z}}^{\ti{x}} (\gm \in du \, | \, \tilde{A} ) = \frac{e^{-2 (y-\tilde{x})^2/t}}{1- e^{-2 \tilde{x}^2/t}}.
    \end{align}
Now consider $\un{N}(t)$ in the first term on the right-hand side of \eqref{eq:split_exp}. Set $\ov{g}(t) := \max_{0 \leq s \leq t}g(s)$ and $\un{g}(t) := \min_{0 \leq s \leq t}g(s)$. Using \eqref{eq:under_N}, assumption \eqref{eq:Q_def} and the assumption that $g(s)$ is twice continuously differentiable, we see that, on the event $\tilde{A}$,
    \begin{align*}
        \un{N}(t) &\leq -c \inf_{0 \leq s \leq t} X_s - \frac{t}{2} \inf_{\inf_{0 \leq s \leq t} X_s + \un{g}(t) < y \leq \ov{g}(t)} \big[\mu'(y)+ \mu^2(y) \big]\\
        & \leq -c \inf_{0 \leq s \leq t} X_s + \frac{t}{2} Q\left(\inf_{0 \leq s \leq t} X_s + \un{g}(t)\right)
    \end{align*}
for some $c>0$, where we have assumed, without loss of generality, that $Q(y)$ is a
decreasing function for $y < 0$. Thus using \eqref{eq:Q_bound}, for large enough $H$
and some $r < 2/t$, we have
    \begin{align*}
        \un{N}(t) &\leq r \left( \inf_{0 \leq s \leq t} X_s \right)^2 \qquad \mbox{ \ on $D$}.
    \end{align*}
Hence, using \eqref{eq:inf_bound},
    \begin{align}
    \label{eq:final_unif_int}
        \ov{\Exp}_{\ti{z}}^{\ti{x}} \left[ e^{\un{N}(t)} \ind_{D }\,\Big|\, \tilde{A}\right] & \leq \int_{-\infty}^{-H} e^{r y^2} d \ov{\Qr}_{\ti{z}}^{\ti{x}} \left( \inf_{0 \leq s \leq t} X_s \leq y \right)\notag\\
        &=\left[ e^{r y^2} \ov{\Qr}_{\ti{z}}^{\ti{x}}\left( \inf_{0 \leq s \leq t} X_s  \leq y \right) \right]_{-\infty}^{-H} - \int_{-\infty}^{-H} 2r y e^{r y^2} \ov{\Qr}_{\ti{z}}^{\ti{x}}\left( \inf_{0 \leq s \leq t} X_s \leq y \right) dy\notag\\
        & \leq \frac{1}{1 - e^{- 2 \tilde{x}^2/t}} \left( e^{r H^2 - 2(H+ \tilde{x})^2/t} - 2 r \int_{-\infty}^{-H} y e^{r y^2 - 2 (y- \tilde{x})^2/t} dy \right),
    \end{align}
which clearly vanishes as $H \rightarrow -\infty$. Thus the expectation in \eqref{eq:final_prob} converges to a finite limit and hence we have
    \[  \Pr_x\left(\sup_{0 \leq s \leq t} (X_s - g(s)) < 0 \,\Big|\, X_t = z\right) = (f(t,x) +o(1))\tilde{z} = (f(t,x) +o(1))(g(t) - z),   \]
completing the proof of Theorem~\ref{th:asymp_holds}.

\end{proof}

\section{Proof of Theorem~\ref{thm:diff_bdy}}
\label{sec:proof_diff}

The proof of this theorem can be divided into four steps. For the first three steps, we assume $h(t) \geq 0$, $0 \leq t \leq 1$. In the first step, we observe that the difference $P(g + \ep h) - P(g)$ can be written as an integral by conditioning on the first crossing time $\tau$ of $g$. In the second step we follow a similar scheme to the proof of Theorem~\ref{th:asymp_holds}, transforming the integrand so it is written as the product of an expectation of a functional of the Brownian meander and a well-known boundary non-crossing probability for the Brownian motion. In the third step, we calculate the limit of the ratio of the thus obtained expression to $\ep$ as $\ep \rightarrow 0$. This involves careful treatment near the right end point of the integration interval. Finally, in the fourth step, we show how to extend the results to general $h$ which are twice continuously differentiable.

\textbf{Step 1:} Clearly, the difference $P(g + \ep h) - P(g)$ is the probability of $(X_s)$ crossing $g$ at some time prior to time $1$ without ever crossing $g + \ep h$ on $[0,1]$ (as we said, we assume here that $h \geq 0$). We condition on the time remaining until time $1$ after the first crossing time of $g$ and use the strong Markov property to obtain that, for $\ep \in (0,1)$,
    \begin{align}
    \label{eq:split_tau}
        P(g + \ep h) - P(g) &= \int_0^1 \Pr_x(1 - \tau \in dt) \Pr_x \left( \sup_{1-t \leq s \leq 1} (X_s - g(s) - \ep h(s)) < 0 \,\Big|\, X_{1-t} = g(1-t) \right)\notag\\
        &= \int_0^1 \Pr_x(1-\tau \in dt) \Pr_{g(1-t)} \left( \sup_{0 \leq v \leq t} (X_v - g(1-t+v) - \ep h(1-t+v)) < 0 \right)\notag\\
        &= \int_0^{\ep^{3/2}} + \int_{\ep^{3/2}}^1 =: J_1 + J_2.
    \end{align}

\textbf{Step 2:} We will show that $J_1 = o(\ep)$ in Step 3. For $J_2$, we continue in a similar manner to the proof of Theorem~\ref{th:asymp_holds}. Note, however, that in this case we do not condition on the end point of the process. Define
    \[  g_{\ep,t}(v) := g(1-t+v) + \ep h(1-t+v), \qquad 0 \leq v \leq t,    \]
and set
    \[  A(t):= \left\{ \sup_{0 \leq v \leq t} (X_v - g_{\ep,t }(v)) < 0) \right \}. \]
Using Girsanov's theorem to change to the Brownian motion measure, we have
    \[  \Pr_{g(1-t)} \left( A(t) \right) = \hat{\Exp}_{g(1-t)} \left[ e^{ G(X_{t}) - G(g(1-t)) + N(t)} \ind_{A(t)} \right], \]
where $\hat{\Exp}_{g(1-t)}$ denotes expectation with respect to the measure $\Qr_{g(1-t)}$ defined in the proof of Theorem~\ref{th:asymp_holds} and $N(t)$ is given by \eqref{eq:Nt_def} (see the proof of Theorem~\ref{th:asymp_holds} for the justification of the use of Girsanov's theorem and \ito's lemma). Again we transform the space using the function $\psi_{\ep,t}:C[0,t] \to C[0,t]$ defined by $\psi_{\ep,t}(f) = f - g_{\ep,t}$. This induces the measure $\ti{\Qr}_{\gd(t)}$ with expectation $\ti{\Exp}_{\gd(t)}$, where $\gd(t) := -\ep h(1-t)$, such that
    \begin{align*}
        \hat{\Exp}_{g(1-t)} &\left[ e^{ G(X_{t}) - G(g(1-t)) + N(t) } \ind_{A(t)} \right] = \ti{\Exp}_{\gd(t)} \left[ e^{ G(X_{t} + g_{\ep,t}(t)) - G(g(1-t)) + \ti{N}_{\ep,t}(0,t)} \ind_{\ti{A}(t)} \right],
    \end{align*}
where
    \[  \ti{N}_{\ep,u}(s,t) := -\frac{1}{2}\int_s^t \left[\mu'(X_v + g_{\ep,u}(v)) + \mu^2(X_v + g_{\ep,u}(v))\right] dv    \]
and
    \begin{align*}
            \ti{A}(t) := \left\{\sup_{0 \leq s \leq t} X_s <0 \right\}.
    \end{align*}
Again we can think of $(X_v)$ under this new measure as driven by the stochastic differential equation
    \[  dX_v = -g'_{\ep,t}(v) dv + d \ti{W}_v,  \]
for a Brownian motion $(\ti{W_v})$. Changing to the measure $\ov{\Qr}_{\gd(t)}$ (with corresponding expectation $\ov{\Exp}_{\gd(t)}$) such that $(X_v)$ is again the Brownian motion gives
    \begin{align*}
        \Pr_{g(1-t)} \left( A(t) \right) = \ov{\Exp}_{\gd(t)} \left[ e^{ G\left( X_{t}+ g_{\ep,t}(t) \right) - G(g(1-t)) - X_{t} g_{\ep,t}'(t) + \ov{N}_{\ep,t}(0,t) } \ind_{\tilde{A}(t)} \right],
    \end{align*}
where
    \[  \ov{N}_{\ep,u}(s,t) := \ti{N}_{\ep,u}(s,t) + \int_s^{t} g''_{\ep,u}(v) X_v dv - \frac{1}{2} \int_s^t (g'_{\ep,u}(v))^2 dv.  \]
We then condition on $\tilde{A}(t)$, which yields
    \begin{align}
    \label{eq:final_form}
        \Pr_{g(1-t)} \left( A(t) \right) = \ov{\Exp}_{\gd(t)} \left[ e^{G\left( X_{t}+ g_{\ep,t}(t) \right) - G(g(1-t)) - X_{t} g_{\ep,t}'(t) + \ov{N}_{\ep,t}(0,t) } \,\Big|\, \tilde{A}(t) \right] \ov{\mathbb{Q}}_{\gd(t)} \left( \tilde{A}(t) \right).
    \end{align}
The last factor is now the probability for the Brownian motion to stay below a fixed level on $[0,t]$. Hence, uniformly in $t \in (\ep^{3/2}, 1)$, we have
    \begin{align}
    \label{eq:Q_bcp}
        \ov{\mathbb{Q}}_{\gd(t)}(\tilde{A}(t)) = 2 \Phi \left( \frac{\ep h(1-t)}{\sqrt{t}} \right) - 1  = \sqrt{\frac{2}{\pi t}} \ep h(1-t) + o\left(\frac{\ep}{\sqrt{t}} \right)
    \end{align}
as $\ep \rightarrow 0$, see e.g.\ (1.1.4) in \cite{Borodin_etal_xx02}, p.~153.

\textbf{Step 3:} We now divide the right-hand side of \eqref{eq:split_tau} by $\ep$ and take the limit as $\ep \rightarrow 0$. Let $c$ denote an upper bound for the density $\Pr_x(1-\tau \in dt)/dt$ on $[0,1]$ (which can be obtained, for example, using Theorem 3.1 of \cite{Downes_etal_xx08}). Then we clearly have
    \[  J_1 \leq c \int_{0}^{\ep^{3/2}} dt = c \ep^{3/2} = o(\ep).  \]
Using \eqref{eq:final_form} and \eqref{eq:Q_bcp}, we see that
    \begin{align}
    \label{eq:pre_lim_dif}
        \frac{J_2}{\ep} &= \sqrt{\frac{2}{\pi}} \int_{\ep^{3/2}}^{1} h(1-t) \ov{\Exp}_{\gd(t)} \Big[ e^{ G\left( X_{t}+ g_{\ep,t}(t) \right) - G(g(1-t)) - X_{t}  g_{\ep,t}'(t) + \ov{N}_{\ep,t}(0,t) } \,\Big|\, \tilde{A}(t) \Big]\notag\\
            & \qquad \qquad\qquad \qquad\qquad \qquad\qquad \qquad \qquad \qquad \times \frac{\Pr_x(1-\tau \in dt)}{\sqrt{t}} (1 + o(1)).
    \end{align}
Using the weak convergence result of Theorem~2.1 in \cite{Durrett_etal1_0277}, together with assumption \eqref{eq:Q2_def} and a similar uniform integrability argument to the one used in the proof of Theorem~\ref{th:asymp_holds} (see \eqref{eq:split_exp}--\eqref{eq:final_unif_int}; that the right-hand side in condition \eqref{eq:Q2_bound} differs from that in \eqref{eq:Q_bound} is due to our dealing with the maximum of the Brownian motion, while in Theorem~\ref{th:asymp_holds} we dealt with that of the Brownian bridge --- which has a thinner distribution tail), we have that, as $\ep \rightarrow 0$, the expectation in \eqref{eq:pre_lim_dif} converges to the expectation under which the process $(-X_s)$ is the Brownian meander on $[0,t]$ (that is, the Brownian motion $(W_s, 0 \leq s \leq t)$ conditioned to remain positive on $[0,t]$). If we denote this process by $\big(W_s^{+, (t)}\big)$, then, by the Brownian scaling, we have
    \begin{align*}
        W_s^{+, (t)} \stackrel{d}{=} \sqrt{t} \, W_{s/t}^{+}.
    \end{align*}
Thus the expectation on the right-hand side of \eqref{eq:pre_lim_dif} converges to the expectation in the statement of Theorem~\ref{thm:diff_bdy}, which completes the proof of Theorem~\ref{thm:diff_bdy} for $h \geq 0$.

\textbf{Step 4:} Now consider a general $h(t)$ satisfying the conditions of the theorem. Using the standard notation $y^- = -\min\{y, 0\}$ for the negative part of $y$, we have
    \begin{align*}
        P(g + \ep h) - P(g) = \big(P(g + \ep h) - P(g - \ep h^-)\big) -\big( P(g) - P(g - \ep h^-) \big).
    \end{align*}
Each of the two terms on the right-hand side can then be evaluated as per Steps 1--3, the ``lower" of the two boundaries (originally it was $g$ due to the assumption $h \geq 0$) now being $g - \ep h^-$. Next we observe that the main terms in the respective expressions for $J_2/\ep$ are continuous functions of $\ep$, the limits of them having the form of the right-hand side of \eqref{eq:deriv_expr} with $h(t)$ replaced with $\max\{h(t), 0\}$ and $\max\{-h(t), 0\}$. This completes the proof of Theorem~\ref{thm:diff_bdy}.

%
%

\section{Examples}
\label{sec:examples}

{\em Example {\em 1}: Brownian Motion and a Linear Boundary.} In this example we will illustrate the results of Theorems~\ref{th:relate_probs} and \ref{thm:diff_bdy} in the case of the Brownian motion $(W_s)$ when $g(s)$ is linear, with $g(0) >x$. Recall that \eqref{eq:bm_lin_bdy} shows that the conditions of Theorem~\ref{th:relate_probs} are satisfied with $f(t,x) = \frac{2}{t} (g(0)-x)$. Thus the theorem implies that
    \begin{align*}
        p_{\tau}(t) &= \frac{1}{2} \frac{2}{t}(g(0)-x) p(t,x,g(t)) = \frac{1}{t}(g(0)-x) \frac{1}{\sqrt{2 \pi t}} e^{-(g(t)-x)^2/(2t)},
    \end{align*}
which is a well known result, being a special case of Kendall's formula \eqref{eq:ken_rel}, see e.g.\ (1.1.4) of \cite{Borodin_etal_xx02}, p.~250.

Next assume that $g(t) = a_1 + b_1 t$ and $h(t) = a_2 + b_2 t$, $0 \leq t \leq 1$, where $b_1, a_2, b_2 \in \mathbb{R}$, and, without loss of generality, $a_1 >0$. Then the assertion of Theorem~\ref{thm:diff_bdy} reduces to
        \begin{align}
        \label{eq:deriv_eg}
            \lim_{\ep \rightarrow 0}  \;\ep^{-1} \Big[&P(g + \ep h ) - P(g) \Big] = \sqrt{\frac{2}{\pi}}  \int_0^{1} \frac{a_2 + b_2(1-t)}{\sqrt{t}}\Pr_x(1-\tau \in dt) \Exp \Big[ e^{ \sqrt{t} b_1 W_{1}^+  -\frac{1}{2} b_1^2 t } \Big].
        \end{align}
On the other hand, we know (see e.g.\ (1.1.4) on p.~250 in \cite{Borodin_etal_xx02}) that, for the boundary $a + b t$, $a,b >0$,
              \begin{equation}
            \begin{array}{rl}
                P(-\infty, a + b t) &= \ov{\Phi} \left(-a -b \right) + e^{-2 ba} \Phi \left( -a + b \right),\\
                \vphantom{.}\\
                            p_{\tau}(t) &= \frac{a}{\sqrt{2 \pi} t^{3/2}} \exp \left\{ -\frac{(a+bt)^2}{2 t} \right\},
            \end{array}
          \label{eq:bm_lin_asymp}
            \end{equation}
while, from e.g.\ (1.1) in \cite{Durrett_etal1_0277}, we have that
    \begin{align}
    \label{eq:mean_dens}
        \Pr (W^+_1 \in dy) = y e^{-y^2/2}   dy, \qquad y>0.
    \end{align}
In the special case $b_1 = b_2 = 0$, we can evaluate both sides of \eqref{eq:deriv_eg} and confirm they each give
    \[  \sqrt{\frac{2}{\pi}} a_2 e^{-a_1^2/2}.  \]
In the general case, we use \eqref{eq:mean_dens} to evaluate the required Laplace transform:
    \begin{align}
    \label{eq:mean_lt}
        \Exp[e^{\gl W_1^+}] = 1 + \sqrt{2 \pi} \gl e^{\gl^2/2} \ov{\Phi}(-\gl), \qquad \gl \in \mathbb{R}.
    \end{align}
Using \eqref{eq:bm_lin_asymp} and \eqref{eq:mean_lt}, in this case \eqref{eq:deriv_eg} is equivalent to the relation
    \begin{align}
    \label{eq:reduce_deriv_eg}
            a_2\sqrt{\frac{2}{\pi}}& e^{-(a_1 + b_1)^2/2} + 2(a_2 b_1 + a_1 b_2)e^{-2 a_1 b_1} \Phi(b_1 - a_1)\notag\\
                & \qquad= \frac{a_1}{\pi} \int_0^{1} \frac{a_2 + b_2(1-t)}{\sqrt{t} (1-t)^{3/2}} \exp \left\{-\frac{(a_1 + b_1 (1-t))^2}{2 (1-t)} -\frac{b_1^2 t}{2}  \right\}  \notag\\
                &\qquad\qquad\qquad\qquad\times \left(1 + \sqrt{2 \pi t} b_1 e^{t b_1^2/2} \ov{\Phi}(-\sqrt{t} b_1) \right) dt.
    \end{align}
It is not immediately clear that this identity holds true. However, for given values of
$a_1$, $a_2$, $b_1$ and $b_2$, we can numerically integrate the right-hand side of
\eqref{eq:reduce_deriv_eg} to confirm \eqref{eq:deriv_eg} holds for these values. For
example, using the values $a_1 = a_2 = b_1 = b_2=1$, both sides of
\eqref{eq:reduce_deriv_eg} give $0.379$. Alternatively, for the values $a_1 = 1$, $a_2
= -0.5$, $b_1 = -1,$ $b_2=2$ (in which case the ``increment'' $h(t)$ assumes values of
both signs on $[0,1]$), both sides of \eqref{eq:reduce_deriv_eg} give $0.442$.

{\em Example {\em 2}: Brownian Motion and a Daniels Boundary.} A general Daniels boundary is given by
    \begin{align}
    \label{eq:dan_bdy}
    g_D(s) := \gd - \frac{s}{2\gd} \log \left( \frac{\gk_1}{2} + \sqrt{\frac{1}{4} \gk_1^2 + \gk_2 e^{-4\gd^2/s} } \right),
  \end{align}
where $\gd \ne 0$, $\gk_1 >0$ and $\gk_2 \in \mathbb{R}$ are subject to $\gk_1^2 + 4 \gk_2 >0$. For this boundary, the first crossing time density is given by
  \begin{align*}
    p_{\tau}(t) = \frac{1}{\sqrt{2 \pi} t^{3/2}} \left( \gd \gk_1 e^{-(g_D(t) - 2\gd)^2/(2t)} + 2\gd\gk_2 e^{-(g_D(t) - 4\gd)^2/(2t)}\right)
  \end{align*}
(see \cite{Daniels_0696} for further information). Theorem~\ref{th:relate_probs} therefore gives
    \begin{align}
    \label{eq:dan_ftx}
        f(t,x) = \frac{2}{t} \left( \gd \gk_1 e^{-\frac{1}{2} (2 \gd -x) (2 \gd + x - 2 g_D(t))}+ 2 \gd \gk_2 e^{-\frac{1}{2t} (4 \gd - x) (4 \gd + x - 2g_D(t))} \right).
    \end{align}

Figure~\ref{fig:sim_cond_cross} shows estimated values of the probability
    \begin{align}
    \label{eq:cond_prob}
        \Pr_x\left(\sup_{0 \leq s \leq t} (W_s - g_D(s)) < 0 \,\Big|\, W_t = z\right),
    \end{align}
using simulation, as a function of $g_D(t) - z$. We also estimate the value of $f(t,x)$ predicted by this simulation (using only the simulated values at the points $0 \leq g_D(t) - z \leq 0.1$) and plot the corresponding linear function for comparison. The parameter values used are $\gd =\gk_1 = \gk_2 = 0.5$, $t=1$ and $x=0$. Simulation of the Brownian bridge process was performed using the fact that if $(B_s)$ is a Brownian motion conditioned on $B_t = z$, $B_0 = x$, then
    \[  B_s \stackrel{d}{=} x + \frac{s}{t} \big( (z-x) - W_t \big) + W_s, \qquad 0 \leq s \leq t,  \]
see e.g.\ \cite{Borodin_etal_xx02}, pp.~64. We used $10^4$ simulations for each value of $g_D(t) - z$; the step size is $10^{-4}$ for $0 \leq t \leq 0.99$ and $10^{-5}$ for $0.99 \leq t \leq 1$ (we use a smaller step size closer to $t=1$ as this region is prone to the most simulation error).
\begin{figure}[!htb]
        \begin{centering}
            \includegraphics{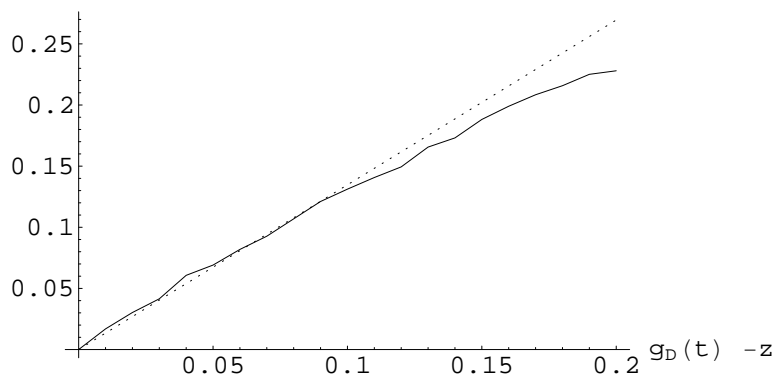}
        \caption{Simulated values of \eqref{eq:cond_prob} as a function of $g_D(t) - z$ for the Daniels boundary \eqref{eq:dan_bdy} with $\gd =\gk_1 = \gk_2 = 0.5$ using $x=0$ and $t=1$ (solid line). The dotted line shows the linear function corresponding to the regression of the values for $0 \leq g_D(t) - z \leq 0.1$.}
    \label{fig:sim_cond_cross}
    \end{centering}
\end{figure}
The regression gives the estimated value of $f(1,0)$ as $1.35$, in comparison \eqref{eq:dan_ftx} gives $f(1, 0) \approx 1.33$. The slightly higher value from simulations is to be expected since simulating continuous stochastic processes using discrete simulation points over-estimates the probability the process will stay below a given boundary, as the simulated process is unable to cross the boundary between two consecutive time steps.

Note that the results of Corollary~\ref{cor:fpt_rels} immediately give the form of the first passage time density for the conditional process in this example.

{\bf Acknowledgements:} This research was supported by the ARC Centre of Excellence for Mathematics and Statistics of Complex Systems, and ARC Discovery Grant DP 0880693.

\end{document}